\documentclass{amsart}

\usepackage{amssymb,amsmath,cite,geometry}

\usepackage{color}

\numberwithin{equation}{section}

\newtheorem{theorem}{Theorem}[section]
\newtheorem{lemma}{Lemma}[section]

\newtheorem{prop}{Proposition}[section]
\newtheorem{remark}{Remark}[section]

\newtheorem{example}{Example}[section]

\allowdisplaybreaks[4]

\begin{document}
\title[Back flow of the unsteady Prandtl boundary layers]
{Back flow of the two-dimensional unsteady Prandtl boundary layer under an adverse pressure gradient}

\author{Y.-G. Wang}
\address{Ya-Guang Wang
\newline\indent
School of Mathematical Sciences, MOE-LSC and SHL-MAC, Shanghai Jiao Tong University,
Shanghai, 200240, P. R. China}
\email{ygwang@sjtu.edu.cn}

\author{S.-Y. Zhu}
\address{Shi-Yong Zhu
\newline\indent
School of Mathematical Sciences, Shanghai Jiao Tong University,
Shanghai, P. R. China}
\email{shiyong\_zhu@sjtu.edu.cn}

\date{}




\begin{abstract} In this paper, we study the back flow of the two-dimensional unsteady Prandtl boundary layer under an adverse pressure gradient. The occurrence of back flow is an important physical event in the evolution of boundary layer, which eventually leads to separation. For the two-dimensional unsteady Prandtl boundary layer equations, when the initial tangential velocity is strictly monotonic with respect to the normal variable, and the pressure gradient of the outer flow is adverse, we obtain that the first critical point of the tangential velocity profile with respect to the normal variable, if exists when the boundary layer evolves in time, must appear on the boundary. Moreover, we give a condition on the growth rate of the initial tangential velocity such that there is a back flow point of the Prandtl boundary layer under the adverse pressure gradient. In the appendix, we introduce two examples showing that back flow occurs either when the flow distance is long in the streamwise direction for a given initial monotonic tangential velocity field, or when the initial tangential velocity grows slowly in a large neighborhood of the boundary for a fixed flow distance in the streamwise direction.
\end{abstract}

\maketitle

\tableofcontents

\section{Introduction}\setcounter{section}{1}\setcounter{equation}{0}
This paper is devoted to the study of the back flow of the two-dimensional unsteady boundary layers under an adverse pressure gradient. Consider the following problem for the Prandtl boundary layer equation with non-slip boundary condition for an unsteady  incompressible flow in a domain $Q_T=\{(t,x,y)|0\le t<T,0\le x\le L, 0\le y<+\infty\}$,
\begin{equation}\label{1.1}
\begin{cases}
\partial_{t}u+u\partial_{x}u+v\partial_{y}u=\partial^2_{y}u-\partial_{x}P,\\
\partial_{x}u+\partial_{y}v=0,\\
u|_{y=0}=v|_{y=0}=0,~\lim\limits_{y\to+\infty} u=U_{e}(t,x),\\
u|_{t=0}=u_{0}(x,y),~u|_{x=0}=u_{1}(t,y),
\end{cases}
\end{equation}
 where $(u(t,x,y),v(t,x,y))$ is the velocity field in the boundary layer, $U_{e}(t,x)$ and $P(t,x)$ are traces at the boundary $\{y=0\}$ of the tangential velocity and pressure of the Euler outer flow respectively, interrelated through Bernoulli's law
\begin{equation}\label{1.2}
\partial_{t}U_{e}+U_{e}\partial_{x}U_{e}=-\partial_{x}P.
\end{equation}

The equation in \eqref{1.1} was first proposed by Prandtl (\!\!\cite{Prandtl}) to describe the behavior of
boundary layers in the small viscosity limit for
the incompressible viscous flow with non-slip boundary condition.
So far, there are some interesting results on well-posedness of problems for the Prandtl equation. In \cite{Oleinik-0}, Oleinik obtained the well-posedness of  the problem  \eqref{1.1} locally in time, for data satisfying
\begin{align}\label{1.3}
u_{0}(x,y)>0,~u_{1}(t,y)>0,~\forall t\in[0,T),x\in [0,L], y\in[0,\infty).
\end{align}
and the monotonicity assumption,
\begin{align}\label{1.4}
~\partial_{y}u_{0}(x,y)> 0,~\partial_{y}u_{1}(t,y)>0,~\forall t\in[0,T),x\in [0,L], y\in[0,\infty).
\end{align}
This result was surveyed in the monograph \cite{Oleinik}. Recently, the local well-posedness result in the monotonic class was also obtained  in the Sobolev spaces in \cite{AWXY} and \cite{MW} by using the energy method. When this monotonicity assumption does not hold for the initial data, there are some interesting results on blowup or instability of solutions to the problem \eqref{1.1} in the Sobolev spaces of solutions, cf. \cite{EE,GD,GN,GuN,KVW} and references therein. The monotonicity assumption of the tangential velocity is believed to be essential for the well-posedness of \eqref{1.1} in the two-dimensional problem, except in the spaces of analytic functions, cf. \cite{IV, KV,LCS,M,SC1, ZZ} or Gevrey functions, cf. \cite{GM,LWX, LY-1, LY-2}.

According to the phenomena observed in fluid mechanics, cf. \cite{Prandtl, S},  when the pressure gradient is favourable, i.e.
the pressure in the outer flow is decreasing in the streamwise direction,
\begin{align}\label{1.5}
\partial_{x} P(t,x)\leq 0, \qquad \forall t>0,~ x\in [0,L].
\end{align}
the Prandtl  boundary layers are expected to be stable globally in time.  Mathematically, Xin and Zhang (\!\!\cite{XZ}) obtained
a global existence of a solution to the two-dimensional Prandtl equation
for monotonic data providing that the pressure is favourable in the sense of \eqref{1.5}.
On the other hand, adverse pressure gradients may lead to the back flow phenomenon of boundary layer, which is an important physical event, eventually leading to separation of boundary layer, because an adverse pressure gradient will retard the fluid in the boundary layer, then there may exist a back flow point, behind which the flow follows the pressure gradient and moves in a direction opposite to the outer flow. The back flow point is defined as $(t_{0},x_{0},0)$ at which
\begin{equation}
\label{1.6}
\partial_{y}u(t_{0},x_{0},0)=0,
\end{equation}
and
\begin{equation}
\label{1.7}
\partial_{y}u(t,x,0)>0, \quad \forall 0<x<x_{0},~ 0<t<t_{0}.
\end{equation}
From mathematical point of view, the monotonicity condition is violated when the back flow point exists, and thus the problem \eqref{1.1} may be ill-posed in  the Sobolev space when $t$ across $t_0$.

For the two-dimensional steady flow, as pointed out in \cite{GO}, the back flow point is the point from which the boundary layer separates from the physical boundary. The first theoretical result was given by Oleinik (\!\!\cite{Oleinik-1}) and Suslov (\!\!\cite{SUS}) on existence of a separation point in the two-dimensional steady Prandtl boundary layer under an adverse pressure gradient. The asymptotic behavior of flow near the separation point of the two-dimensional steady Prandtl equations was formally investigated by Goldstein (\!\!\cite{GO}), and improved by Stewartson in \cite{Ste}.  A rigorous analysis of this asymptotic behavior of flow near separation was first studied by E and Caffarelli in an unpublished manuscript mentioned in \cite{E}, then was obtained in detail recently by Dalibard and Masmoudi in \cite{DM,DM-1}, and by Shen, Wang and Zhang in \cite{SWZ}. 

However, Moore (\!\!\cite{Moore}), Rott (\!\!\cite{Rott}) and Sears (\!\!\cite{Sears}) pointed out that the back flow point defined as in \eqref{1.6}-\eqref{1.7}, in general, is not a separation point in unsteady flows, and they concluded that separation occurs at the point of zero shear stress within the boundary layer, rather than on the surface as in the steady case, and it is a singular point of flow, see the review articles \cite{CC,ST}. Since then, many people studied singularities in boundary layers, cf. \cite{CS,GG,GS-1,GS-2,vDS} and references therein. To our knowledge, till now there does not exist any mathematical theory on back flow points of unsteady boundary layers, which certainly is important itself, and also is the first step to study the appearance of separation point. The goal of this paper is to consider the back flow for the problem \eqref{1.1} of the unsteady Prandtl equation, when the data satisfy the monotonicity assumption \eqref{1.4} but with an adverse pressure gradient.

The main result of this paper is as follows:

\begin{theorem}
(1) Assume that the trace at the boundary of the outer Euler flow satisfies
\begin{equation}\label{1.8}
U_{e}\in C^{1}([0,T]\times [0,L]) ~\mbox{and}~  U_e(t,x)>0,~\forall t\in[0,T],~x\in [0,L],
\end{equation}
and the uniformly adverse pressure gradient in the sense that
\begin{equation}\label{1.9}
\partial_{x} P(t,x)>0, \quad \forall t\in [0,T],~ x\in [0,L].
\end{equation}
Let $(u,v)$ be the local classical solution to the problem \eqref{1.1} corresponding to the data $u_{0},u_{1}$ satisfying \eqref{1.3} and the monotonicity condition \eqref{1.4}. Then, the first zero point of $\partial_{y}u(t,x,y)$, when the time evolves, should be at the boundary $\{y=0\}$ if it exists for some time $t>0$.

(2) Moreover, when the initial velocity $u_0(x,y)$ satisfies
\begin{align}\label{1.10}
\int_{0}^{\infty}\!\!\int_0^L\frac{(L-x)^{\frac{3}{2}}\partial_{y}u_{0}}{\sqrt{(\partial_{y}u_{0})^{2}+u_{0}^{2}}}dxdy\geq C_{*},
\end{align}
for a positive constant $C_{*}$ depending only on $L,T,U_{e}$ and $\partial_{x}P$,
then there is a back flow point $(t^*, x^*)\in (0, T)\times [0,L]$, such that
\begin{equation}\label{1.11}
\begin{cases}
\partial_yu(t^*, x^*, 0)=0,\\
\partial_yu(t, x, y)>0, \quad \forall 0<t<t^*, x\in [0,L], y\ge 0.
\end{cases}
\end{equation}
Moreover, we have
$\partial^2_yu(t^*, x^*, 0)\neq 0$.

\end{theorem}

To obtain the results stated in the above theorem, we shall deduce the first result given in Theorem 1.1 by using a contradiction argument and the maximum principle for a scalar degenerate parabolic problem derived from \eqref{1.1} by employing the Crocco transformation. Then, the existence of a back flow point will be obtained by using a Lyapunov functional argument. 

This paper is organized as follows. In Section 2, we investigate the possible location of the first critical point of the tangential velocity profile in the Prandtl boundary layer. Then, we obtain the existence of a back flow point under the assumption \eqref{1.10} in Section 3. Finally, in Section 4, we introduce examples to show that the condition \eqref{1.10} of admitting back flow point holds under the hypothesis on the largeness of the space interval $[0,L]$ or on the growth rate of the initial velocity $u_0(x,y)$ in $y$ for a fixed space interval.

\section{Position of the first critical point of the tangential velocity}

In this section, as the time evolves, we shall prove that the first possible critical point of the tangential velocity $u(t,x,y)$ with respect to $y$ should be on the boundary $\{y=0\}$.

Since the initial data given in \eqref{1.1} are strictly monotonic as given in \eqref{1.4}, there is a local classical solution to the problem \eqref{1.1} for $0\le t<T$, in the class of $\partial_yu>0$ for all $y\ge 0$, as obtained in \cite{Oleinik-0, Oleinik} and \cite{AWXY, MW}. In this monotonic class, as in \cite{Oleinik}, the following Crocco transformation is invertible,
\begin{align}\label{2.1}
\tau=t, ~\xi=x, ~\eta=\frac{u(t,x,y)}{U_{e}(t,x)},
\end{align}
and with \eqref{2.1},
\begin{align}\label{2.2}
w(\tau,\xi,\eta)=\frac{\partial_{y}u(t,x,y)}{U_{e}(t,x)},
\end{align}
satisfies the following initial boundary value problem in
$Q^{*}_T=\{(\tau,\xi,\eta)|0\le\tau<T,~\xi\in [0,L],~\eta\in[0,1]\}$,
\begin{equation}\label{2.3}
\begin{cases}
\partial_{\tau}w+\eta U_{e}\partial_{\xi}w+A\partial_{\eta}w+Bw=w^{2}\partial^2_{\eta}w,\\
w\partial_{\eta}w|_{\eta=0}=\frac{\partial_{\xi}P}{U_{e}},~w|_{\eta=1}=0,\\
w|_{\tau=0}=w_{0}:=\frac{\partial_y u_0}{U_{e}},~w|_{\xi=0}=w_{1}:=\frac{\partial_y u_1}{U_{e}},
\end{cases}
\end{equation}
where $A=(1-\eta^{2})\partial_{\xi}U_{e}+(1-\eta)\frac{\partial_{\tau}U_{e}}{U_{e}}$, and $B=\eta\partial_{\xi}U_{e}+\frac{\partial_{\tau}U_{e}}{U_{e}}$.

In this section, we shall have the following result:

\begin{prop}
Under the same assumption as given in Theorem 1.1(1), let $(u,v)$ be a local classical solution to the problem \eqref{1.1}, then the first critical point of $u(t,x,y)$ with respect to $y$, if exists, can only be at the boundary $\{y=0\}$.
\end{prop}

To prove this proposition, we first have the following lemma for the problem \eqref{2.3}.

\begin{lemma}
Let $w$ be a local classical solution to the problem \eqref{2.3} for $0\le \tau<T$. Under the same assumption as given in Proposition 2.1, there is a constant $C_{0}$ depending only on $U_{e}$ and $T$, such that we have
\begin{equation}\label{2.3-1}
\sup\limits_{Q^*_T}w^{2}(\tau, \xi, \eta)\leq C_{0}\max\left(\sup_{0\leq \xi\leq L, 0\leq \eta\leq 1}w_{0}^{2}(\xi, \eta), \sup_{0\leq \tau<T, 0\leq \eta \leq 1}w_{1}^{2}(\tau, \eta)\right).
\end{equation}
\end{lemma}

\begin{proof}
For the solution $w$ of the problem \eqref{2.3}, set $f(\tau,\xi,\eta):=e^{-N\tau}w^{2}$
with $N$ satisfying $N+2B\geq0$. Then, from \eqref{2.3}, we know that the function $f(\tau,\xi,\eta)$ satisfies
the following problem  in $Q^{*}_T$,
\begin{equation}\label{2.4}
\begin{cases}
\partial_{\tau}f+\eta U_{e}\partial_{\xi}f+A\partial_{\eta}f+(N+2B)f
=w^{2}\partial^2_{\eta}f-2e^{-N\tau}w^{2}(\partial_{\eta}w)^{2},\\
\partial_{\eta}f|_{\eta=0}=\frac{2\partial_{\xi}P}{U_{e}}e^{-N\tau},~f|_{\eta=1}=0,\\
f|_{\tau=0}=w_{0}^{2},~f|_{\xi=0}=w_{1}^{2}e^{-N\tau}.
\end{cases}
\end{equation}
From the adverse pressure gradient condition \eqref{1.9}, we know that $\partial_\eta f$ is positive at $\{\eta=0\}$, thus $f$ can not attain its supremum in $Q^*_T$ at $\{\eta=0\}$. Applying the maximum principle of parabolic equations for the problem \eqref{2.4}, we deduce that $f$ can attain its supremum in $Q^*_T$ at $\{\tau=0\}$ or $\{\xi=0\}$, from which we get the conclusion \eqref{2.3-1} immediately.
\end{proof}

\begin{proof}[Proof of Proposition 2.1]
We prove this proposition by a contradiction argument. Assume the first critical point of $u$ with respect to $y$ is an inner point $(t_{0},x_{0},y_{0})$ with $0<t_{0}< T,~0\le x_{0}\leq L,~0<y_{0}<\infty$. Then for any $0\le t<t_{0}$, the Crocco transformation given in \eqref{2.1} is invertible.
For a sufficiently small $\epsilon$, by continuity of $w$ in $\tau$, there must exist a point $\frac{t_0}{2}<\tau_{\epsilon}<t_{0}$ near $t_{0}$ such that
\begin{equation}\label{2.5-0}
0<w(\tau_{\epsilon},\xi_{0},\eta_{0})< \epsilon^{2},
\end{equation}
where $\xi_{0}=x_{0}$ and $\eta_{0}=\frac{u(\tau_{\epsilon},\xi_{0},y_{0})}{U_{e}(\tau_{\epsilon},\xi_{0})}$.
Moreover, we have
\begin{align}\label{2.5}
w(\tau,\xi,\eta)>0 \quad \mbox{in}~\{(\tau,\xi,\eta)|0\leq\tau\leq \tau_{\epsilon},~0\leq\xi\leq \xi_{0},~0\leq\eta<1\}.
\end{align}

(1) Consider a function $F(\tau,\xi,\eta)$ defined in
$$D_\epsilon=\{(\tau,\xi,\eta)|0\leq\tau\leq \tau_{\epsilon},~0\leq\xi\leq \xi_{0},~0\leq\eta\leq1\}$$
as
\begin{align*}
F(\tau,\xi,\eta):=w(\tau,\xi,\eta)-\epsilon\phi(\eta)e^{-M\tau},
\end{align*}
where parameters $M$ and $\epsilon$ will be determined later, $\phi\in C^{2}([0,1])$ satisfies
\begin{align*}
\phi(\eta)=
\begin{cases}
\frac{1}{2\eta_{1}}\eta, ~&\eta\in [0,\eta_{1}],\\
1,~&\eta=\eta_{0},\\
1-\eta,~&\eta\in [\eta_{2}, 1],
\end{cases}
\end{align*}
for some $\eta_{1}, \eta_{2}\in(0,1)$ satisfying $\eta_{1}<\eta_{0}<\eta_{2}$,
moreover, $\partial_{\eta}\phi\geq0$ for any $\eta\in[0,\eta_{0}]$, and
$\partial_{\eta}\phi\leq0$ for any $\eta\in[\eta_{0},1]$.

From \eqref{2.3}, we know that $F$ satisfies the following problem
in $D_\epsilon$,
\begin{equation}\label{2.5-1}
\begin{cases}
\partial_{\tau}F+\eta U_{e}\partial_{\xi}F+A\partial_{\eta}F+BF
=w^{2}\partial^2_{\eta}F+\epsilon\mathcal{F},\\
F|_{\eta=0}=w(\tau,\xi,0),~
F|_{\eta=1}=0,~\\
F|_{\tau=0}=w_{0}-\epsilon\phi(\eta),F|_{\xi=0}=w_{1}-\epsilon\phi(\eta)e^{-M\tau},
\end{cases}
\end{equation}
where
\begin{align*}
\mathcal{F}
=M\phi(\eta)e^{-M\tau}-A\partial_{\eta}\phi(\eta)e^{-M\tau}-B\phi(\eta)e^{-M\tau}
+w^{2}\partial^2_{\eta}\phi(\eta)e^{-M\tau}.
\end{align*}

(2) We claim that $\mathcal{F}\geq0$ in $D_\epsilon$ when $M$ is properly large. To this end, divide the interval $[0,1]$ of $\eta$ into three parts, $[0,1]=[0,\eta_{1})\cup[\eta_{1},\eta_{2}]\cup(\eta_{2},1]$ and study $\mathcal{F}$ in these subintervals respectively.

i) When $\eta\in[0,\eta_{1})$, from the definition of $\phi$ we know
\begin{align*}
-A\partial_{\eta}\phi(\eta)e^{-M\tau}-B\phi(\eta)e^{-M\tau}
=&-\frac{1}{2\eta_{1}}\left[\partial_{\xi}U_{e}+\frac{\partial_{\tau}U_{e}}{U_{e}}\right]e^{-M\tau}\\
=&\frac{1}{2\eta_{1}}\frac{\partial_{\xi}P}{U_{e}}e^{-M\tau}>0,
\end{align*}
by using the assumption \eqref{1.9}.

Noticing that $\partial^2_{\eta}\phi(\eta)=0$ on $[0,\eta_{1})$, we  have
\begin{align*}
\mathcal{F}(\tau, \xi, \eta)
>M\phi(\eta)e^{-M\tau}\geq 0, \quad \forall \eta\in[0,\eta_1)
\end{align*}
for any non-negative $M$.

ii) When $\eta\in[\eta_{1},\eta_{2}]$, by using Lemma 2.1 we know that $A\partial_{\eta}\phi(\eta)+B\phi(\eta)
-w^{2}\partial^2_{\eta}\phi(\eta)$ is bounded. Since $\phi(\eta)\geq\min\{\frac{1}{2},1-\eta_{2}\}$ on $[\eta_{1},\eta_{2}]$, we can choose $M$ large enough such that
\begin{align*}
M\phi(\eta)-A\partial_{\eta}\phi(\eta)-B\phi(\eta)+w^{2}
\partial^2_{\eta}\phi(\eta)\ge 0,
\end{align*}
which implies that $\mathcal{F}(\tau,\xi,\eta)\ge 0$ in $[\eta_1,\eta_2]$.

iii) Noticing that $A=(1-\eta)\left[(1+\eta)\partial_{\xi}U_{e}+\frac{\partial_{\tau}U_{e}}{U_{e}}\right]$, for $\eta\in(\eta_{2},1]$ we have
\begin{align*}
A\partial_{\eta}\phi=-\left[(1+\eta)\partial_{\xi}U_{e}+\frac{\partial_{\tau}U_{e}}{U_{e}}\right]\phi.
\end{align*}
Moreover, $\partial^2_{\eta}\phi(\eta)=0$ when $\eta\in(\eta_{2},1]$. Thus we have
\begin{align*}
\mathcal{F}
=\left[M+(1+\eta)\partial_{\xi}U_{e}+\frac{\partial_{\tau}U_{e}}{U_{e}}-B\right]\phi(\eta)e^{-M\tau}
=\left(M+\partial_{\xi}U_{e}\right)\phi(\eta)e^{-M\tau}
\end{align*}
which implies $\mathcal{F}\geq0$ on $(\eta_2, 1]$, by choosing $M$ properly large such that $M+\partial_{\xi}U_{e}\geq0$. From now on, we fix $M$ large such that $\mathcal{F}(\tau,\xi,\eta)\geq0$ in the whole $D_{\epsilon}$.

(3) To apply for the maximum principle for the problem \eqref{2.5-1},
choose another constant $N$ properly large such that $N+B\geq0$.
Denote by $G(\tau,\xi,\eta):=e^{-N\tau}F(\tau,\xi,\eta)$. From \eqref{2.5-1}, we know that $G(\tau,\xi,\eta)$
satisfies the following problem in $D_\epsilon$,
\begin{equation}\label{2.6}
\begin{cases}
\partial_{\tau}G+\eta U_{e}\partial_{\xi}G+A\partial_{\eta}G+(N+B)G
=w^{2}\partial^2_{\eta}G+\epsilon\mathcal{F}e^{-N\tau},\\
G|_{\eta=0}=w(\tau,\xi,0)e^{-N\tau},~
G|_{\eta=1}=0,~\\
G|_{\tau=0}=w_{0}-\epsilon\phi(\eta),G|_{\xi=0}=\left[w_{1}-\epsilon\phi(\eta)e^{-M\tau}\right]e^{-N\tau},
\end{cases}
\end{equation}

Choose $\epsilon_1>0$ small enough, such that when $0<\epsilon\le \epsilon_1$,
\begin{equation}\label{2.7}
  G(\tau_{\epsilon},\xi_{0},\eta_{0})<\left(\epsilon^{2}-\epsilon e^{-M\tau_{\epsilon}}\right)e^{-N\tau_{\epsilon}}<0,
\end{equation}
and
\begin{equation}\label{2.7-1}
  1-e^{-(M+N)\tau_\epsilon}+\epsilon e^{-N\tau_\epsilon}<1
\end{equation}
hold.
Set $\tilde{\eta}=\max(\eta_2, 1-e^{-(M+N)\tau_\epsilon}+\epsilon e^{-N\tau_\epsilon})$, and let $0<\epsilon_2\le \epsilon_1$ be small
such that when $0<\epsilon\le \epsilon_2$, we have
\begin{align}\label{2.8}
  \epsilon\leq\min\limits_{\xi\in[0,\xi_{0}],\eta\in(0,\tilde{\eta}]}\frac{w_{0}(\xi,\eta)-\epsilon^{2}e^{-N\tau_{\epsilon}}}{\phi(\eta)},\quad
  \epsilon\leq\min\limits_{\tau\in[0,\tau_{\epsilon}],\eta\in(0,\tilde{\eta}]}\frac{w_{1}(\tau,\eta)-\epsilon^{2}e^{-N\tau_{\epsilon}}}{\phi(\eta) e^{-M\tau}},
\end{align}

From \eqref{2.7} and \eqref{2.8}, we get immediately that
\begin{align}\label{2.9}
\min_{\xi\in[0,\xi_{0}],\eta\in[0,\tilde{\eta}]}G(0,\xi,\eta)\ge \epsilon^2e^{-N\tau_\epsilon}>G(\tau_{\epsilon},\xi_{0},\eta_{0}),
\end{align}
and
\begin{align}\label{2.10}
\min_{\tau\in[0,\tau_{\epsilon}],\eta\in[0,\tilde{\eta}]}G(\tau,0,\eta)\ge \epsilon^2e^{-N\tau_\epsilon}>G(\tau_{\epsilon},\xi_{0},\eta_{0}).
\end{align}

On the other hand, when $\eta\in [\tilde{\eta}, 1]$, obviously from \eqref{2.7-1}, we have
\begin{align}\label{2.9-1}
w_0(\xi,\eta)-\epsilon\phi(\eta)
\ge \epsilon(\tilde{\eta}-1)\geq(\epsilon^2-\epsilon e^{-M\tau_\epsilon})e^{-N\tau_\epsilon}>G(\tau_{\epsilon},\xi_{0},\eta_{0}),
\end{align}
and
\begin{align}\label{2.10-1}
(w_1(\tau,\eta)-\epsilon\phi(\eta)e^{-M\tau})e^{-N\tau}
\ge \epsilon(\tilde{\eta}-1)\geq(\epsilon^2-\epsilon e^{-M\tau_\epsilon})e^{-N\tau_\epsilon}>G(\tau_{\epsilon},\xi_{0},\eta_{0}),
\end{align}
when $0<\epsilon\le \epsilon_2$.

Combining \eqref{2.9}, \eqref{2.10} with \eqref{2.9-1}, \eqref{2.10-1} respectively, it follows that
\begin{align}\label{2.15}
\min_{\xi\in[0,\xi_{0}],\eta\in[0,1]}G(0,\xi,\eta)>G(\tau_{\epsilon},\xi_{0},\eta_{0}),
\end{align}
and
\begin{align}\label{2.16}
\min_{\tau\in[0,\tau_{\epsilon}],\eta\in[0,1]}G(\tau,0,\eta)>G(\tau_{\epsilon},\xi_{0},\eta_{0}),
\end{align}
hold.

Applying the maximum principle in the problem \eqref{2.6}, and by using \eqref{2.15}-\eqref{2.16} it follows that
\begin{align*}
\min_{\tau\in[0,\tau_{\epsilon}],\xi\in[0,\xi_{0}]}w(\tau,\xi,0)e^{-N\tau}\leq G(\tau_{\epsilon},\xi_{0},\eta_{0})<0.
\end{align*}
Since $\min\limits_{\xi\in[0,\xi_{0}]}w(0,\xi,0)>0$, there must exist a point $(\tau_{*},\xi_{*})\in(0,\tau_{\epsilon})\times(0,\xi_{0}]$, such that
$w(\tau_{*},\xi_{*},0)=0$. This is in contradiction with \eqref{2.5}. Therefore, we conclude the assertion given in Proposition 2.1.

\end{proof}

\section{Existence of a back flow point}

In this section, under the assumptions given in Theorem 1.1(2), we shall prove there exists a critical point of the tangential velocity $u(\cdot, y)$ of \eqref{1.1} with respect to $y$ at the boundary $\{y=0\}$. It will be obtained by a contradiction approach. From Proposition 2.1, we know that under the monotonicity condition \eqref{1.4} and the adverse pressure gradient assumption \eqref{1.9}, the first zero point of $\partial_yu(\cdot,y)$, if exists, could not be an interior of $0\le y<+\infty$. From now on, we assume that $\partial_yu(t,x,y)$ is positive everywhere in the domain
$$Q_T=\{(t,x,y)|0\le t<T, 0\le x\le L, 0\le y<+\infty\},$$
for the problem \eqref{1.1} under the assumptions \eqref{1.4} and \eqref{1.9}. Then, the Crocco transformation \eqref{2.1} is invertible in $Q_T$, and $w(\tau, \xi, \eta)=\frac{\partial_yu}{U_{e}}>0$ in
$$Q^*_T=\{(\tau,\xi,\eta)|0\le \tau<T, 0\le \xi\le L, 0\le \eta<1\}.$$

Denote by $W(\tau,\xi,\eta)=\mathcal{W}^{-\frac{1}{2}}(\tau,\xi,\eta)$, with
$\mathcal{W}(\tau,\xi,\eta)=w^{2}+\eta^{2}$. From \eqref{2.3}, we know that $W>0$ satisfies the following problem in $Q^{*}_T$:
\begin{equation}\label{3.1}
\begin{cases}
\partial_{\tau}W+\eta U_{e}\partial_{\xi}W+A\partial_{\eta}W
=BW
-\frac{w^{3}}{\mathcal{W}^{\frac{3}{2}}}\partial^2_{\eta}w
+\eta\frac{\partial_{\xi}P}{U_{e}}W^{3},\\
(\partial_{\eta}W+\frac{\partial_{\xi}P}{U_{e}}W^{3})|_{\eta=0}=0,~
W|_{\eta=1}=1,\\
W|_{\tau=0}=(w_{0}^{2}+\eta^{2})^{-\frac{1}{2}},W|_{\xi=0}=(w_{1}^{2}+\eta^{2})^{-\frac{1}{2}}.
\end{cases}
\end{equation}

For the problem \eqref{3.1}, we have
\begin{prop}
Under the same assumptions as given in Theorem 1.1(2), there exists a point $(\tau^{*},\xi^{*})\in (0,T)\times [0,L]$, such that \begin{equation}\label{3.2}
W(\tau^{*},\xi^{*},0)=\infty.
\end{equation}
\end{prop}

Obviously, from \eqref{3.2} we have $w(\tau^{*},\xi^{*},0)=0$, which is in contradiction with $w>0$ in $Q^{*}_{T}$.
We shall prove this proposition by constructing a Lyapunov functional and concluding that this functional blows up within the time interval $(0,T)$, provided that the initial data of $W$ is suitable large.

\begin{proof}[Proof of Proposition 3.1]

Let $\varphi(\xi)=(L-\xi)^{\frac{3}{2}}$, denote by
\begin{align*}
\mathcal{G}(\tau)=\int_{\Omega}W(\tau,\xi,\eta)\varphi(\xi)d\xi d\eta,
\end{align*}
where $\Omega=[0,L]_\xi\times[0,1)_\eta.$

From \eqref{3.1}, we know
\begin{align}\label{3.3}
\frac{d}{d\tau}\mathcal{G}
=&-\int_{\Omega}\eta U_{e}\partial_{\xi}W\varphi d\xi d\eta
-\int_{\Omega}(A\partial_{\eta}W-BW)\varphi d\xi d\eta
+\int_{\Omega}\eta\frac{\partial_{\xi}P}{U_{e}}W^{3}\varphi d\xi d\eta\\
&-\int_{\Omega}\frac{w^{3}}{\mathcal{W}^{\frac{3}{2}}}\partial^2_{\eta}w\varphi d\xi d\eta\nonumber\\
=&\sum_{i=1}^{4}\mathcal{R}_{i},\nonumber
\end{align}
with obvious notations $\mathcal{R}_{i}(1\leq i\leq 4)$. Now we shall estimate each $\mathcal{R}_{i}$ step by step.

i) By using integration by parts, we have
\begin{align}\label{3.4}
\mathcal{R}_{1}
=&\int_{\Omega}\eta \partial_{\xi}U_{e}W\varphi d\xi d\eta
+\int_{\Omega}\eta U_{e}W\partial_{\xi}\varphi d\xi d\eta
+C_{0}(\tau),
\end{align}
where $$C_{0}(\tau)=L^{\frac{3}{2}}U_{e}(\tau,0)\int_{0}^{1}\frac{\eta}{\sqrt{w_{1}^{2}+\eta^{2}}}d\eta>0.$$
For the second term on the right hand side of \eqref{3.4}, we use the Young inequality to obtain
\begin{align*}
\int_{\Omega}\eta U_{e}W\partial_{\xi}\varphi d\xi d\eta
\geq& -\int_{\Omega}\eta\left[\frac{U_{e}^{4}}{\partial_{\xi}P}\right]^{\frac{1}{2}}d\xi d\eta
-\frac{1}{2}\int_{\Omega}\eta\frac{\partial_{\xi}P}{U_{e}}W^{3}\varphi d\xi d\eta\\
=& -C_{1}(\tau)
-\frac{1}{2}\int_{\Omega}\eta\frac{\partial_{\xi}P}{U_{e}}W^{3}\varphi d\xi d\eta,\nonumber
\end{align*}
where
\begin{align*}
C_{1}(\tau)=\frac{1}{2}\int_0^L\left[\frac{U_{e}^{4}}{\partial_{\xi}P}\right]^{\frac{1}{2}}d\xi >0.
\end{align*}
Thus, from \eqref{3.4} we have
\begin{align}\label{3.5}
\mathcal{R}_{1}
\geq& \int_{\Omega}\eta \partial_{\xi}U_{e}W\varphi d\xi d\eta
-\frac{1}{2}\int_{\Omega}\eta\frac{\partial_{\xi}P}{U_{e}}W^{3}\varphi d\xi d\eta
+C_{0}(\tau)-C_{1}(\tau).
\end{align}

ii) Recall that $A=(1-\eta^{2})\partial_{\xi}U_{e}+(1-\eta)\frac{\partial_{\tau}U_{e}}{U_{e}}$ and $B=\eta\partial_{\xi}U_{e}+\frac{\partial_{\tau}U_{e}}{U_{e}}$. By using integration by parts we have,
\begin{align}\label{3.6}
\mathcal{R}_{2}
=&\int_{\Omega}\partial_{\eta}AW\varphi d\xi d\eta
-\int_0^L\frac{\partial_{\xi}P}{U_{e}}W(\tau,\xi,0)\varphi d\xi
+\int_{\Omega}BW\varphi d\xi d\eta\\
=&-\int_{\Omega}\eta\partial_{\xi}U_{e}W\varphi d\xi d\eta
-\int_0^L\frac{\partial_{\xi}P}{U_{e}}W(\tau,\xi,0)\varphi d\xi.\nonumber
\end{align}

iii) By using the H\"{o}lder inequality,
\begin{align*}
\int_{\Omega}\eta\frac{\partial_{\xi}P}{U_{e}}W^{3}\varphi d\xi d\eta
\geq 2C_{2}(\tau)\left(\int_{\Omega}W\varphi d\xi d\eta\right)^{3}=2C_{2}(\tau)\mathcal{G}^{3},
\end{align*}
where
\begin{align*}
C_{2}(\tau)=\frac{1}{2}\left(2\int_0^L\left[\frac{U_{e}}{\partial_{\xi}P}\right]^{\frac{1}{2}}\varphi d\xi\right)^{-2}>0.
\end{align*}
Thus we have
\begin{align}\label{3.7}
\mathcal{R}_{3}
\geq \frac{1}{2}\int_{\Omega}\eta\frac{\partial_{\xi}P}{U_{e}}W^{3}\varphi d\xi d\eta
+C_{2}(\tau)\mathcal{G}^{3},
\end{align}

iv) Finally, we use integration by parts and the Young inequality to obtain
\begin{align}\label{3.8}
\mathcal{R}_{4}
=&3\int_{\Omega}\frac{w^{2}}{\mathcal{W}^{\frac{3}{2}}}(\partial_{\eta}w)^{2}\varphi d\xi d\eta
-3\int_{\Omega}\frac{w^{4}}{\mathcal{W}^{\frac{5}{2}}}(\partial_{\eta}w)^{2}\varphi d\xi d\eta\\
&-3\int_{\Omega}\frac{\eta w^{3}}{\mathcal{W}^{\frac{5}{2}}}\partial_{\eta}w\varphi d\xi d\eta
+\int_0^L\frac{\partial_{\xi}P}{U_{e}}W(\tau,\xi,0)\varphi d\xi\nonumber\\
=& 3\int_{\Omega}\frac{\eta^{2}w^{2}}{\mathcal{W}^{\frac{5}{2}}}(\partial_{\eta}w)^{2}\varphi d\xi d\eta
-3\int_{\Omega}\frac{\eta w^{3}}{\mathcal{W}^{\frac{5}{2}}}\partial_{\eta}w\varphi d\xi d\eta+\int_0^L\frac{\partial_{\xi}P}{U_{e}}W(\tau,\xi,0)\varphi d\xi\nonumber\\
\geq& -\frac{3}{4}\int_{\Omega}\frac{w^{4}}{\mathcal{W}^{\frac{5}{2}}}\varphi d\xi d\eta
+\int_0^L\frac{\partial_{\xi}P}{U_{e}}W(\tau,\xi,0)\varphi d\xi\nonumber\\
\geq& -\frac{3}{4}\mathcal{G}+\int_0^L\frac{\partial_{\xi}P}{U_{e}}W(\tau,\xi,0)\varphi d\xi\nonumber.
\end{align}

Combining \eqref{3.5}-\eqref{3.8} with \eqref{3.3} we obtain
\begin{align}\label{3.9}
\frac{d}{d\tau}\mathcal{G}\geq C_{2}(\tau)\mathcal{G}^{3}-\frac{3}{4}\mathcal{G}+C_{0}(\tau)-C_{1}(\tau).
\end{align}

Since for any $(\tau, \xi)\in [0,T)\times [0,L]$, $\partial_{\xi}P$ is bounded below away from zero as well as $U_{e}$, moreover $U_{e}$ is bounded above, there exists positive constants $\lambda_{0}$, $\lambda_{1}$ and $\lambda_{2}$ such that $C_{0}(\tau)\geq\lambda_{0}, C_{1}(\tau)\leq\lambda_{1} , C_{2}(\tau)\geq\lambda_{2}$. Thus, from \eqref{3.9} it follows
\begin{align}\label{3.10}
\frac{d}{d\tau}\mathcal{G}\geq \lambda_{2}\mathcal{G}^{3}-\frac{3}{4}\mathcal{G}+\lambda_{0}-\lambda_{1}.
\end{align}

Therefore, if the initial data of $\mathcal{G}$ is properly large, that is when the assumption \eqref{1.10} holds, $\mathcal{G}(\tau)$ blows up within the time interval $(0,T)$. That means there must exist a point $(\tau^{*},\xi^{*},\eta^{*})\in(0,T)\times [0,L]\times[0,1)$, such that $W(\tau^{*},\xi^{*},\eta^{*})=\infty$. Recall that $W(\tau,\xi,\eta)=\left(w^{2}+\eta^{2}\right)^{-\frac{1}{2}}$, we get $\eta^{*}=0$ and $w(\tau^{*},\xi^{*},0)=0$.
\end{proof}

By combining Proposition 2.1 with Proposition 3.1, we get there is a separation point $(t^*, x^*, 0)$ with $(t^*, x^*)\in (0,T)\times [0, L]$ for the tangential velocity profile $u(t,x,y)$ to the problem \eqref{1.1}.
  Moreover, from the first equation given in  \eqref{1.1} we know that $\partial^2_{y}u=\partial_{x}P>0$ at this separation point, thus this critical point of $u$ in $y$ is non-degenerate.

\section{Appendix: Examples of back flow in boundary layers}

In this section, we shall give two examples on the existence of back flow points of boundary layers, from which we know that the sufficient condition \eqref{1.10} holds when the space interval $[0,L]$ is properly large or when the initial velocity $u_0(x,y)$ grows slowly with respect to $y$ in a large neighborhood of the boundary $\{y=0\}$ for a fixed $L>0$.

\begin{example}

{\rm
Consider the tangential velocity of the outer flow being given
$$U_{e}(t,x)=e^{-L^{5}t}(2L-x), \quad 0\le x\le L.$$
From the Bernoulli law, we have
$$\partial_{x}P(t,x)=L^{5}e^{-L^{5}t}(2L-x)+e^{-2L^{5}t}(2L-x)>0, \quad \forall~0\le x\le L.$$

Choose the initial tangential velocity $u_0(x,y)$ of the problem \eqref{1.1} being monotonic in $y\ge 0$ such that
\begin{equation}\label{5.1}
\int_{0}^{\infty}\frac{\partial_{y}u_{0}}{\sqrt{(\partial_{y}u_{0})^{2}+u_{0}^{2}}}dy\geq c_{0},~\forall~x\in [0,L],
\end{equation}
holds for a fixed constant $c_0>0$.

For the solution of \eqref{1.1} with the above data, let $w(\tau,\xi,\eta)$ be determined by
using the Corocco transformation \eqref{2.1}-\eqref{2.2}.
From the computation given in Section 3,
we know that
$$\mathcal{G}(\tau)=\int_0^1\!\!\int_0^L(w^2(\tau, \xi,\eta)+\eta^2)^{-\frac{1}{2}}(L-\xi)^{\frac{3}{2}}d\xi d\eta$$
satisfies the following inequality
\begin{align*}
\frac{d}{d\tau}\mathcal{G}
\geq\frac{25}{32}\mathcal{G}^{3}-\frac{3}{4}\mathcal{G}-\frac{4\sqrt{2}-1}{5},
\end{align*}
which implies that $\mathcal{G}(\tau)$ will blow up in a finite time when $\mathcal{G}(0)\ge C_0$ for a positive constant $C_0$ independent of $L$.

On the other hand, by using the Crocco transformation and \eqref{5.1} we have
\begin{align*}
\mathcal{G}(0)  =
\int_{0}^{\infty}\!\!\int_0^L\frac{(L-x)^{\frac{3}{2}}\partial_{y}u_{0}}{\sqrt{(\partial_{y}u_{0})^{2}+u_{0}^{2}}}dxdy
\geq c_{0}\int_{0}^{L}(L-x)^{\frac{3}{2}}dx=\frac{2}{5}c_{0}L^{\frac{5}{2}},
\end{align*}
which is large than or equal to $C_0$ when $L$ is properly large. Thus, by using Theorem 1.1
there is a separation point in this case.
}

\end{example}

\begin{remark} In contrast to one well-posedness result given in Oleinik and Samokhin \cite[Theorem 4.2.3]{Oleinik}, in which they obtained the global existence of a classical solution to the problem \eqref{1.1} when $L>0$ is small, the above example shows that for a general monotonic initial datum satisfying \eqref{5.1}, the Prandtl boundary layer usually will flow back in a finite time when the space interval $[0,L]$ is properly large.

\end{remark}

\begin{example}

{\rm Fix $L=1$, and assume that the tangential velocity of the Euler outer  flow on the boundary is the following one,
$$U_{e}(x)=2-x.$$
From the Bernoulli law, we have
$$\partial_{x}P(x)=2-x, \quad \forall 0\le x\le 1.$$
Choose the initial data of the problem \eqref{1.1} being
$u_{0}(x,y)=U(x)\phi(y)$, where $\phi(y)$ satisfies
\begin{equation*}
\begin{cases}
\phi(y)=\alpha y\qquad \mbox{for}~y\leq M,\\
\lim\limits_{y\to\infty}\phi(y)=1,\\
\partial_{y}\phi>0\quad \mbox{for}~y\in[0,\infty),\\
\int_{0}^{\infty}\frac{\phi'(y)}{\phi'(y)+\phi(y)}dy<+\infty,
\end{cases}
\end{equation*}
for some $M\gg 1$ to be determined later, with $0<\alpha<M^{-1}$ being fixed. For the solution of \eqref{1.1} with the above data, let $w(\tau,\xi,\eta)$ be determined by using the Corocco transformation \eqref{2.1}-\eqref{2.2}.
From the computation given in \S3, we know that
$$\mathcal{G}(\tau)=\int_0^1\!\!\int_0^1(w^2(\tau, \xi,\eta)+\eta^2)^{-\frac{1}{2}}(1-\xi)^{\frac{3}{2}}d\xi d\eta$$
satisfies the following inequality
\begin{align*}
\frac{d}{d\tau}\mathcal{G}
\geq\frac{25}{32}\mathcal{G}^{3}-\frac{3}{4}\mathcal{G}-\frac{4\sqrt{2}-1}{5}
\end{align*}
which implies that $\mathcal{G}(\tau)$ will blow up in a finite time when $\mathcal{G}(0)\ge C_0$ for a positive constant $C_0$.
On the other hand, it is easy to have
\begin{align*}
\mathcal{G}(0)=\int_{0}^{\infty}\!\!\int_0^1\frac{(1-x)^{\frac{3}{2}}\partial_{y}u_{0}}{\sqrt{(\partial_{y}u_{0})^{2}+u_{0}^{2}}}dxdy
\geq \frac{2}{5}\int_{0}^{M}\frac{1}{\sqrt{1+y^2}}dy
\end{align*}
which is large than or equal to $C_0$ when $M$ is properly large. Thus, by using Theorem 1.1
there exists separation in this case.

}

\end{example}

\begin{remark} This example shows that for a fixed flow distances $L$ in the streamwise direction, the boundary layer shall flow back when the initial velocity $u_{0}(x,y)$ grows slowly with respect to $y$ in a large neighborhood of the boundary $\{y=0\}$, and separation occurs earlier as the neighborhood larger.

\end{remark}

\noindent{\bf Acknowledgments:} The authors would like to express their gratitude to Weinan E and Fanghua Lin for their valuable discussion on this topic, and to Tong Yang for the suggestion of introducing Example 5.1 in \S5.   This research was partially supported by
National Natural Science Foundation of China (NNSFC) under Grant No. 11631008, and by Shanghai Committee of Science and Technology under Grant No. 15XD1502300.

\end{document}